\documentclass[11pt]{article}
\usepackage[utf8]{inputenc}
\usepackage{geometry}

\usepackage{graphicx}
\graphicspath{ {./images/} }
\usepackage{amsmath} 
\usepackage{amssymb} 
\usepackage{listings}
\usepackage[dvipsnames]{xcolor}
\usepackage{color}
\usepackage[vcentermath]{youngtab}
\usepackage{amsthm}
\usepackage{tikz}
\usepackage{lscape}
\usepackage{hyperref}
\usepackage{adjustbox}
\usepackage{authblk, float}
\usepackage[linesnumbered,lined,commentsnumbered,ruled]{algorithm2e}

\newcommand\msmall[1]{\mbox{\small\ensuremath{#1}}}

\usepackage{my_preamble}

\usepackage[compact]{titlesec}
\titlespacing{\section}{0pt}{1ex}{1ex}
\hypersetup{
    colorlinks=true,
    linkcolor=magenta,
    filecolor=magenta,      
    urlcolor=magenta,
    citecolor=magenta
}
\title{Logarithmic Voronoi polytopes for discrete linear models}
\author{Yulia Alexandr}
\affil[]{University of California, Berkeley}
\date{}

\begin{document}
\maketitle
\begin{abstract}
We study logarithmic Voronoi cells for linear statistical models and partial linear models. The logarithmic Voronoi cells at points on such model are polytopes. To any $d$-dimensional linear model inside the probability simplex $\Delta_{n-1}$, we can associate an $n\times d$ matrix $B$. For interior points, we describe the vertices of these polytopes in terms of cocircuits of $B$. We also show that these polytopes are combinatorially isomorphic to the dual of a vector configuration with Gale diagram $B$. This means that logarithmic Voronoi cells at all interior points on a linear model have the same combinatorial type. We also describe logarithmic Voronoi cells at points on the boundary of the simplex. Finally, we study logarithmic Voronoi cells of partial linear models, where the points on the boundary of the model are especially~of~interest.
\end{abstract}

\section{Introduction}\label{intro}
Given $X\subset \RR^n$, the \textit{Voronoi cell} at a point $p\in X$ is defined to be the set of all points in $\RR^n$ that are closer to $p$ than to any other point in $X$ with respect to the Euclidean metric. When $X$ is a finite set, Voronoi cells at all points in $X$ tessellate $\RR^n$ into convex polyhedra. When $X$ is a variety, the Voronoi cells of $X$ are convex semialgebraic sets in the normal space of $X$, and their algebraic boundary was computed in \cite{Voronoi-cells-of-varieties}. One can study the properties of real algebraic varieties, such as their Voronoi cells, that depend on the distance metric. This is the objective of metric algebraic geometry \cite{DiRoccoEklundWeinstein2020BottleneckDegreeOfAlgebraicVarieties, SturmfelsEuclideanDistanceDegree2016, maddiesthesis}.

Logarithmic Voronoi cells arise naturally in this context when we replace the Euclidean distance with Kullback-Leibler divergence of probability distributions \cite[Section 2.7]{AyJostLeSchwachhofer2017InformationGeometryTEXT}. Given a statistical model $\M$ inside the probability simplex $\Delta_{n-1}$, the \textit{maximum likelihood estimator} (MLE) maps an empirical distribution $u\in\Delta_{n-1}$ to the point $p$ on the model that best explains the data: $p$ maximizes the log-likelihood function $\ell_u(x)=\sum_{i=1}^nu_i\log(x_i)$. Thus, the logarithmic Voronoi cell at $p\in \M$ is the set of all data points $u\in\Delta_{n-1}$ for which $\ell_u(x)$ is maximized at $p$, or equivalently, at which Kullback-Leibler divergence from $u$ to the model is minimized. Logarithmic Voronoi cells, first introduced in \cite{AH20logarithmic}, are convex sets that tessellate the probability simplex. As such, they lie at the interface of geometric combinatorics and algebraic statistics.

Discrete linear models are given by the intersection of an affine linear space with the probability simplex. These models are important in statistics and particle physics \cite{modulispaces}, and their logarithmic Voronoi cells are polytopes. After giving the basic definitions in Section \ref{prelim}, we describe these polytopes combinatorially in Section \ref{combinatorialtype}. Proposition \ref{vertices-are-cocircuits} gives a formula for the vertices of logarithmic Voronoi cells at interior points. We show that logarithmic Voronoi cells at all interior points on a linear model have the same combinatorial type and describe this combinatorial type using Gale diagrams and polar duals in Theorem \ref{Gale-diagrams}. 

In Section \ref{boundary}, we focus on the points on a linear model that lie on the boundary of the simplex. Theorem \ref{linear-boundary} gives a combinatorial description of logarithmic Voronoi cells at such points for linear models in general position. In particular, the moduli spaces $\M_{0,m}$, studied in \cite{modulispaces}, can be viewed as $(m-3)$-dimensional linear models inside the simplex $\Delta_{n-1}$ where $n=m(m - 3)/2$. We compute logarithmic Voronoi cells at the boundary points of $\M_{0,6}$ in Example \ref{moduli}. Finally, Section \ref{partial} investigates logarithmic Voronoi cells of partial linear models. A partial linear model is given by a polytope, but such that not all facets of the polytope lie on the boundary of the simplex. Every partial linear model can be extended to a linear model. In Theorem \ref{partial-interior}, we show that for a partial linear model, logarithmic Voronoi cells at points in the interior of the polytope agree with logarithmic Voronoi cells of its linear extension. Theorem \ref{partial-facets} describes logarithmic Voronoi cells at the points in the relative interior of the facets of the model.

\section{Definitions}\label{prelim}
In this paper we study partitions of the probability simplex
$$\Delta_{n-1}=\left\{u\in\RR^n: \sum_{i=1}^nu_i=1, u_i\geq 0\text{ for all }i\in[n]\right\}$$
into combinatorially equivalent polytopes. These partitions are induced by linear statistical models. In Figure \ref{figure:tetrahedra-examples-d=1}, such models are given by the dotted line segments.

A \textit{statistical model} $\M$ is defined to be any subset of $\Delta_{n-1}$. When $\M$ is the intersection of an algebraic variety with the simplex, $\M$ is said to be \textit{algebraic}. When the variety is a linear space, we say the model is \textit{linear}.  We often work with the open probability simplex, denoted by $\Delta_{n-1}^{\circ}$, which is defined to be the interior of $\Delta_{n-1}$. For a point $u\in\Delta_{n-1}^\circ$, we define the \textit{log-likelihood function} $\ell_u:\RR_{>0}^n\to \RR$ as $\ell_u(x)=\sum_{i=1}^nu_i\log x_i$ \cite{solv-lik-eq}. For any model $\M\subseteq\Delta_{n-1}$, define the relation $\Phi$ on $\M\times\Delta_{n-1}$ by
$$(u,p)\in\Phi\iff p\in\argmax_{q\in\M}\ell_u(q).$$
This relation is known as the \textit{likelihood correspondence} \cite[Definition 1.5]{Huh2014}.
If $(u,p)\in\Phi$, we write $\Phi(u)=p$ and say that $p$ is the \textit{maximum likelihood estimate} (MLE) of $u$ in $\M$. For arbitrary models, describing the set of those points $u$ for which MLE exists is an active area of research, and we refer the reader to \cite{AmendolaKohnReichenbachSeigal2020InvariantTheoryScalingAlgorithmsForMaximumLikelihoodEstimation, ErikssonFienbergRinaldoSullivant2006PolyhedralConditionsForNonexistenceOfMLE, Fienberg1970Quasi-independenceMaximumLikelihoodEstimationIncompleteContingencyTables, GrossRodriguez2014MaximumLikelihoodPresenceDataZeros}. Given a point $u\in\Delta_{n-1}$, the log-likelihood function $\ell_u(x)$ is strictly concave on the simplex and hence on any convex subset of the simplex. Both linear and partial linear models are given by polytopes, so the MLE will always exist and be unique for every $u\in\Delta_{n-1}^\circ$.

\begin{figure}[H]
    \centering
    \includegraphics[width=0.45\textwidth]{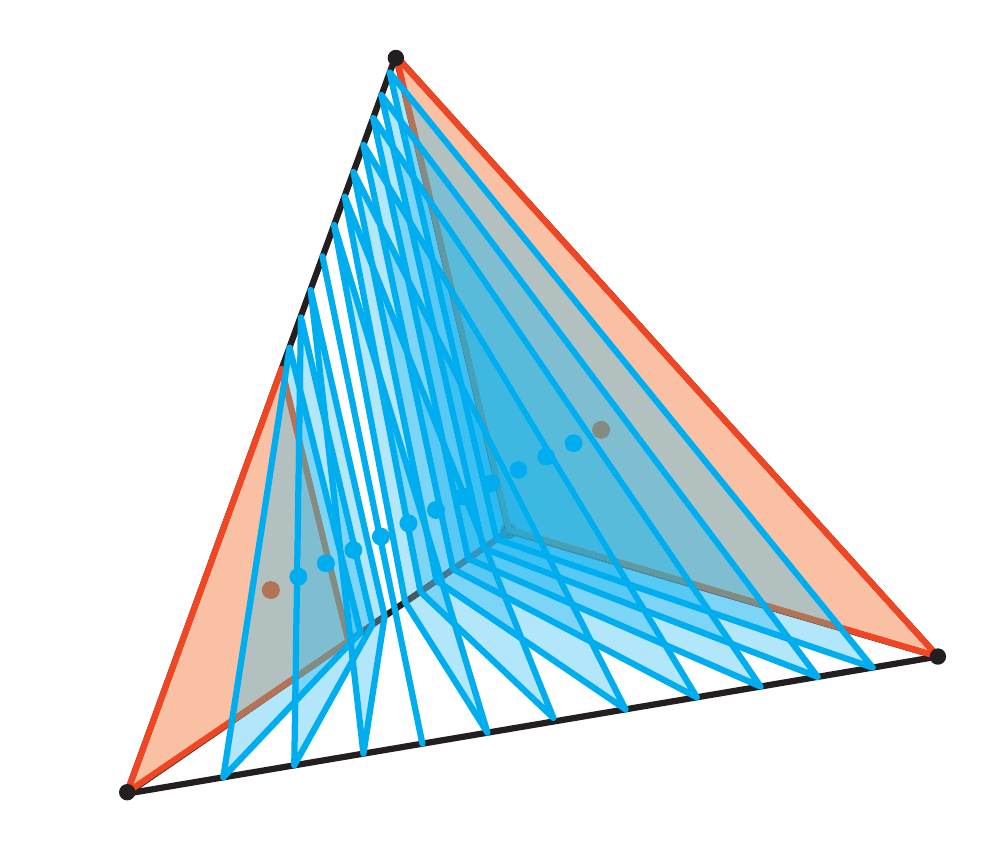} \includegraphics[width=0.42\textwidth]{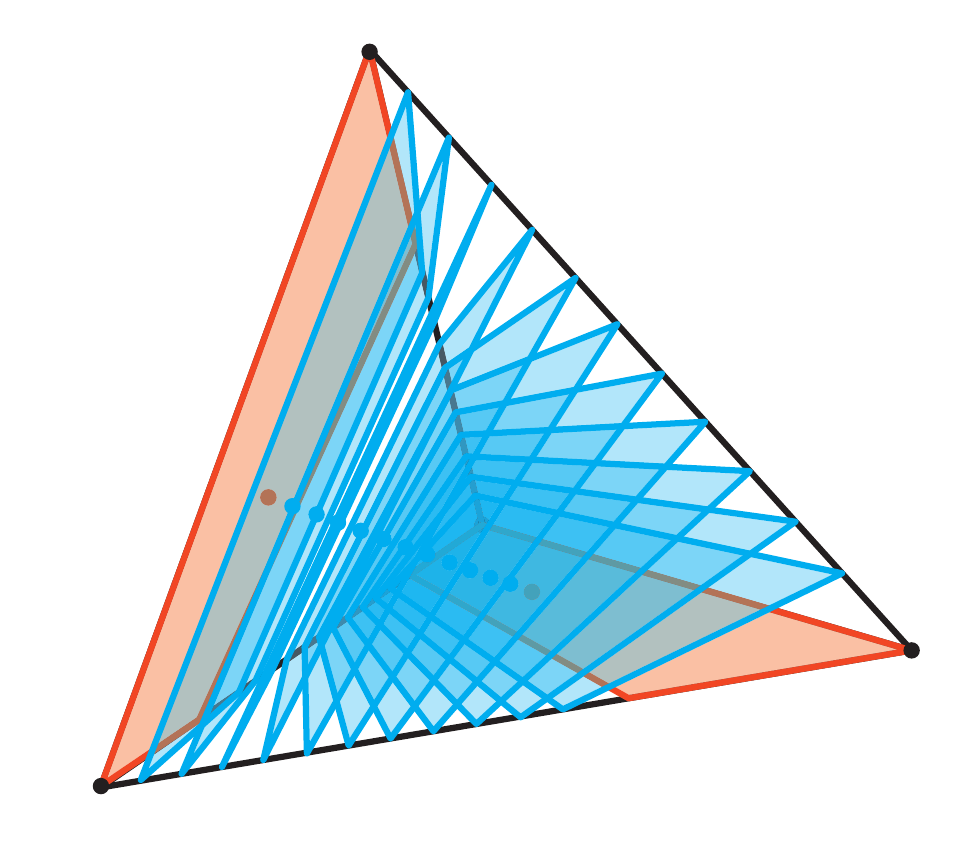} 
    \caption{Partition of the tetrahedron $\Delta_{3}$ into triangles (left) and quadrilaterals (right).}
    \label{figure:tetrahedra-examples-d=1}
\end{figure}

\begin{definition}
Given a model $\M$ and a point $p\in\M$, the \textit{logarithmic Voronoi cell} at $p$ is defined to be the set of all $u\in\Delta_{n-1}$ such that $\Phi(u)=p$.
\end{definition} 

In general, logarithmic Voronoi cells are convex sets. For some models, logarithmic Voronoi cells are known to be polytopes; for example, finite models, models of ML degree one, toric models, and linear models \cite{AH20logarithmic}. For such models, we will refer to their logarithmic Voronoi cells as \textit{logarithmic Voronoi polytopes}. For discrete linear models, these polytopes partition the probability simplex.

\begin{definition}\label{defn-linear-model}
Let $f_1,\ldots, f_n$ be linear polynomials in $\theta=(\theta_1,\ldots,\theta_d)$ with $\sum_{i=1}^nf_i(\theta)=1$. Let $\Theta\subseteq \RR^d$ be the $d$-dimensional parameter space where $f_i(\theta)>0$ for all $\theta\in\Theta$ and $i=1,\ldots,n$. The image of the map
$$f:\Theta\to \Delta_{n-1}: \theta\mapsto (f_1(\theta),\ldots,f_n(\theta))$$
is called a \textit{discrete linear model} \cite[Section 1.2]{PachterSturmfels},\cite[Section 7.2]{Sullivant2018AlgebraicStatistics}.
\end{definition} 
Note that a linear model is a polytope, obtained by intersecting an affine linear space with a probability simplex. The dimension of a linear model is the dimension of the corresponding linear space.

Fix a linear model $\M$ as in the definition. For any point $p\in\M$, we denote the logarithmic Voronoi polytope at $p$ by $\log\Vor_{\M}(p)$. We call a point $p=(p_1,\cdots,p_n)\in\M$ \textit{interior} if $p_i > 0$ for all $i\in[n]$. Our hypotheses in Definition \ref{defn-linear-model} imply that any linear model can be written as
$$\M=\{c-Bx: x\in\Theta\}$$
where $B$ is a $n\times d$ matrix, each of whose columns sums to 0, and $c\in\RR^n$ is a vector, whose coordinates sum to 1.

\section{Combinatorial types of logarithmic Voronoi polytopes}\label{combinatorialtype}

In this section, we give an explicit combinatorial description of logarithmic Voronoi polytopes at interior points on the linear model $\M$. We find that these polytopes have the same combinatorial type. The next proposition gives a formula for the vertices of those~polytopes. Slightly abusing the notation in \cite[Chapter 6]{ziegler2012lectures}, we define a \textit{cocircuit} of the matrix $B$ to be any vector $v\in\RR^n$ of minimal support such that $vB = 0$.

\begin{prop}\label{vertices-are-cocircuits}
For any interior point $p\in\M$, the vertices of $\log\Vor_{\M}(p)$ are of the form $v\cdot\diag(p)$ where $v$ is any positive cocircuit of $B$ such that $\sum_{i=1}^nv_ip_i=1$.
\end{prop}

\begin{proof}
The log-likelihood function of a point $u\in\Delta_{n-1}$ is
$$\ell_u(x)=\sum_{i=1}^nu_i\log(c_i-b_i\cdot x)$$
where $b_1,\ldots,b_n$ are the rows of $B$. The likelihood equations \cite[Chapter 2]{DrtonSturmfelsSullivant2009LecturesOnAlgebraicStatistics} have the form
$$\frac{u_1}{p_1}\cdot b_{1j}+\cdots+\frac{u_n}{p_n}\cdot b_{nj}=0\hspace{2em}\text{ for each $j\in[d]$}.$$
Since the log-likelihood function $\ell_u$ is strictly concave, it has a unique critical point on the model. Thus, the points in the logarithmic Voronoi polytope at $p$ are the distributions $u$ on which likelihood equations vanish. Equivalently, $\log\Vor_{\M}(p)$ is the set of all distributions that satisfy the linear equations $\left(u_1/p_1,u_2/p_2,\ldots, u_n/p_n\right)\cdot B=0$. Hence, we may write:
\begin{align*}
  \log\Vor_{\M}(p)&=\{u\in\Delta_{n-1}: u\cdot\text{diag}(p)^{-1}B=0\}\\
  &=\left\{r\cdot\text{diag}(p)\in\RR^n: rB=0,\; r\geq 0,\; \sum_{i=1}^nr_ip_i=1\right\}.
\end{align*}
Now consider an $n\times (d+1)$ matrix $M$ obtained from $B$ by adjoining $(p_1,\ldots,p_n)^T$ as the first column. Then the logarithmic Voronoi polytope at $p$ can be identified with the feasible region of a linear program, namely $r^TM=(1,0,\ldots,0), r\geq 0$. From the simplex method \cite[Chapter 3]{simplex-method-ref}, we know that the vertices of such polytope are the basic feasible solutions, i.e. minimal support vectors in the region. Those basic solutions are precisely the positive cocircuits $v$ of $B$ for which $\sum_{i=1}^nv_ip_i=1$. Since $p$ is interior, $v\cdot\diag(p)$ has the same support as $v$. Thus the vertices of the logarithmic Voronoi polytope at $p$ are precisely the points $v\cdot\diag(p)$ where $v$ is a positive cocircuit of $B$ for which $\sum_{i=1}^nv_ip_i=1$.
\end{proof}

Now we describe logarithmic Voronoi cells at interior points combinatorially. We use the formalism of Gale diagrams, as described in \cite[Chapter 6]{ziegler2012lectures}. For two polytopes $P_1, P_2$, we will write $P_1\sim P_2$ to mean that $P_1$ and $P_2$ are combinatorially equivalent. We will denote the polar dual of a polytope $P$ by $P^{\Delta}$.

Given our linear model $\M=\{c-Bx: x\in\Theta\}$, note that the configuration ${b_1,\ldots,b_n}$ of row vectors of $B$ is totally cyclic, i.e. $\sum_{i=1}^nb_i=0$, since each column of $B$ sums to 0. Hence, $B$ is a Gale transform of some affine configuration
of $n$ vectors $\{v_1, \ldots, v_n\}$ in $\RR^{n-d-1}$ \cite[Section 6.4]{ziegler2012lectures}. Since a Gale transform uniquely
determines the configuration up to an affine transformation, we may assume that $0 \in
\conv\{v_1, \ldots, v_n\}$. Note that this configuration is not necessarily in convex position; however, its
dual is a polytope. This polytope will have the same combinatorial type as the logarithmic Voronoi
cells at interior points of $\M$, as shown in the next theorem.

\begin{thm}\label{Gale-diagrams}
For any interior point $p$ of the linear model $\M$, the logarithmic Voronoi polytope at $p$ is combinatorially equivalent to the dual of the polytope obtained by taking the convex hull of a vector configuration with Gale diagram $B$.
\end{thm}
\begin{proof}
As discussed above, let $\{v_1,\ldots,v_n\}$ be a vector configuration whose Gale diagram is $B$. Let $P=\conv\{v_1,\ldots,v_n\}$ and assume $0\in P$. We wish to show that $\log\Vor_\M(p)\sim P^{\Delta}$. 
Define
$$Q:=\left\{r\in\RR^n: rB=0,\; r\geq 0,\; \sum_{i=1}^nr_i=1\right\}.$$
Then $\log\Vor_\M(p)\sim Q$, since multiplication by $\text{diag}(p)$ is an affine transformation and does not change the combinatorial type of $\log\Vor_{\M}(p)$. 
It then suffices to show $Q\sim P^{\Delta}$. Let $V$ be the matrix whose column vectors are $v_1,\ldots,v_n$, and let
$$A=\begin{bmatrix}1&1&\cdots&1\\
v_1&v_2&\cdots&v_n\end{bmatrix}.$$ We may assume that the rows of $A$ are linearly independent. Since the Gale diagram of $B^T$ is $A^T$, we know that $\ker B^T=\im A^T$. 

We will show $Q\sim P^{\Delta}$. Observe
\begin{align*}
    Q=\left\{r\in\RR^n:rB=0,\; r\geq0,\; \sum_{i=1}^nr_i=1\right\}
    &=\left\{xA\in\RR^{n}: xA\geq0,\; \sum_{i=1}^n(xA)_i=1\right\}\\
    &\sim\left\{x\in\RR^{n-d}: xA\geq0,\; \sum_{i=1}^n(xA)_i=1\right\}.
\end{align*}
The last equivalence follows from the fact that the rows of $A^T$ are linearly independent. Therefore, the cone over $Q$ is $C(Q)=\left\{x\in\RR^{n-d}: xA\geq0\right\}.$

For the dual of the polytope $P$, we may write:
\begin{align*}
    P^{\Delta}=\{z\in\RR^{n-d-1}: V^Tz\leq 1\}
    &=\{z\in\RR^{n-d-1}:(1,-z)A\geq 0\}\\
    &\sim \{x\in\RR^{n-d}: x_1=1,\; xA\geq 0\}.
\end{align*}
Hence the cone over $P^{\Delta}$ is also $\left\{x\in\RR^{n-d}: xA\geq0\right\}.$ Note that this is a pointed cone at the origin, and both polytopes $Q$ and $P^{\Delta}$ are obtained by intersecting this cone with a hyperplane that doesn't contain the origin. Hence, all the extreme rays are intersected by both hyperplanes. It follows that $Q$ and $P^{\Delta}$ have the same combinatorial type by \cite[Proposition 2.4]{ziegler2012lectures}. Therefore, $\log\Vor_{\M}(p)$ is indeed combinatorially equivalent to $P^{\Delta}$.
\end{proof}

\begin{cor}
The Logarithmic Voronoi polytopes at all interior points in a linear model have the same combinatorial type.
\end{cor}

\begin{example}
The points $\{v_1,\ldots,v_n\}$ in the previous statement need not be in convex position, but the dual of their configuration is. For example, consider a $1$-dimensional linear model inside the $3$-simplex, given by $c=(1/4,1/4,1/4,1/4)^T$ and $B=[1\; -5 \;\; 2\;\; 2]^T$. The parameter space is the interval $\Theta=[-1/20, 1/8]\subseteq \RR$ and the model is parametrized 
$$\Theta\to\Delta_{3},\;\;\;x\mapsto(-x+1/4, 5x+1/4, -2x+1/4, -2x+1/4).$$
The $4\times1$ matrix $B$ is a Gale transform of the non-convex configuration $\{(-1,-1), (1,1),\\(3,0), (0,3)\}$. Its convex hull is the triangle $\conv\{(-1,-1), (3,0), (0,3)\}$, a self-dual polytope. The logarithmic Voronoi polytope at any interior point $p=c-Bx$ is also a triangle, with the vertices
\begin{align*}
    \frac{1}{7}(0,\;40x + 2,\;0,\;-40x + 5),
    \frac{1}{7}(0,\;40x + 2,\;-40x + 5,\;0),
    \frac{1}{6}(-20x + 5,\;20x + 1,\; 0,\; 0).
\end{align*}
for the corresponding parameter $x\in(-1/20, 1/8)$.  This is demonstrated in Figure \ref{figure:ex-non-convex}.

\begin{figure}[H]
    \centering
    \includegraphics[width=0.4\textwidth]{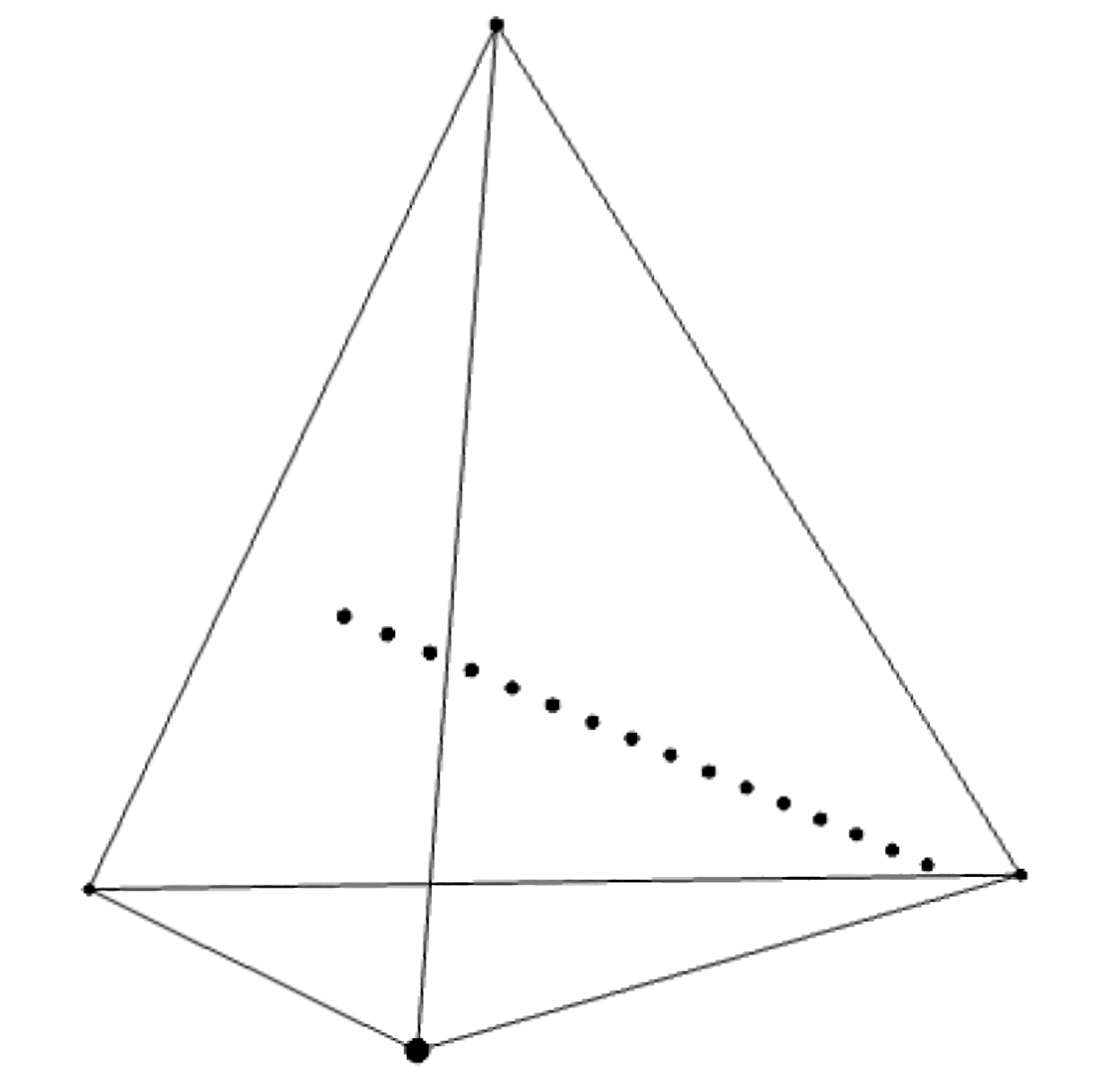}\hspace{2em}
    \includegraphics[width=0.4\textwidth]{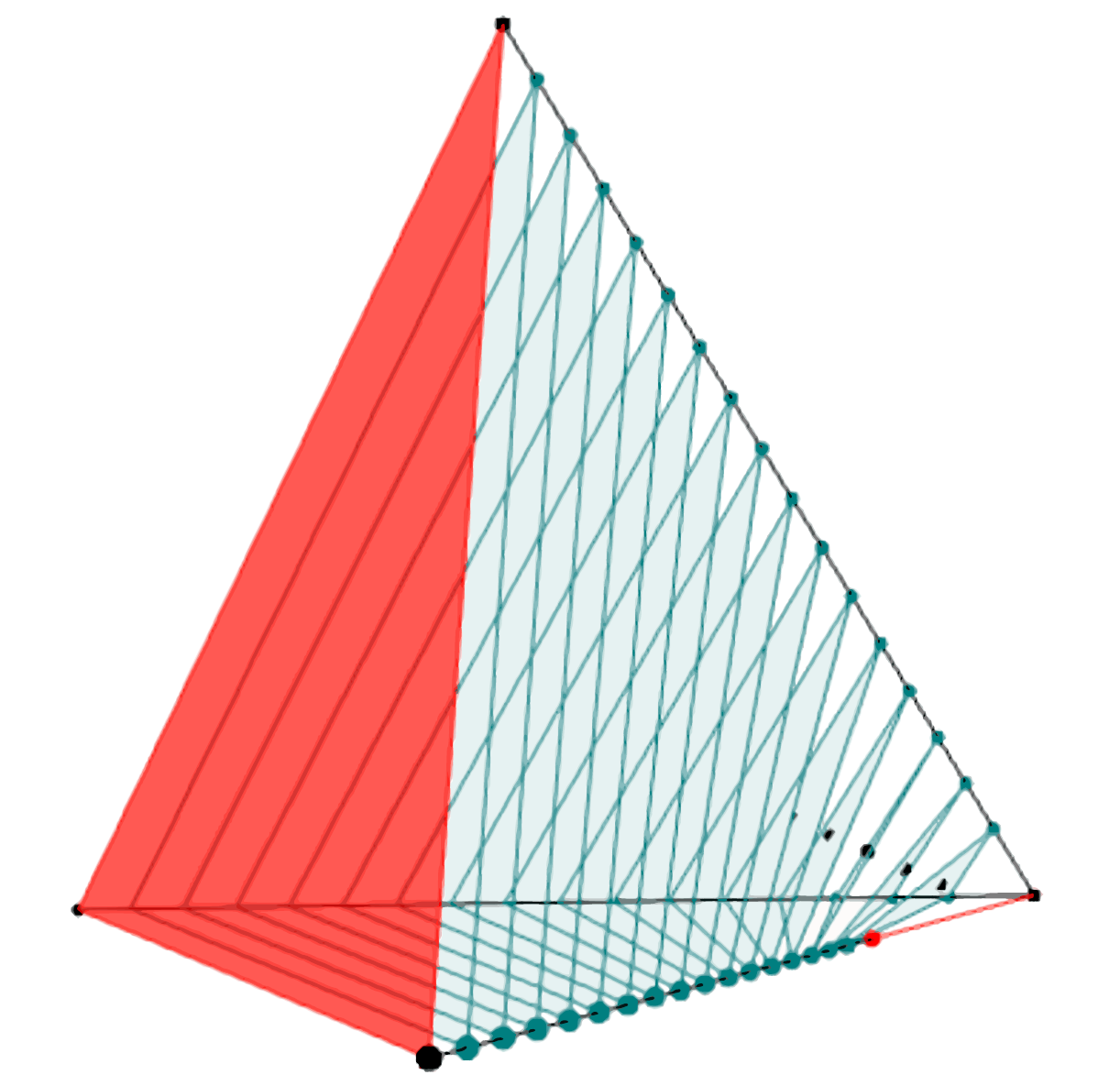}
    \caption{Sampled points on the linear model corresponding to $B=[1\,-5\;2\;2]$ and triangular logarithmic Voronoi cells.}
    \label{figure:ex-non-convex}
\end{figure}

If we take $B$ to be the matrix $[1\;\;5\;\;-3\;\;-3]^T$, which is a Gale diagram of a convex 4-gon, the logarithmic Voronoi polytopes at the interior points on this model would be quadrilaterals in $\Delta_3$. So, (a dual of) any 2-dimensional convex polytope on 4 vertices is a logarithmic Voronoi polytope for some 1-dimensional model in $\Delta_3$. In fact, this holds~in~general.
\end{example}

\begin{prop}\label{every-polytope-appears}
Every $(n-d-1)$-dimensional polytope with at most $n$ facets appears as a logarithmic Voronoi polytope of a $d$-dimensional linear model inside $\Delta_{n-1}$.
\end{prop}

\begin{proof}
Let $P$ be a polytope of dimension ${n-d-1}$ with at most $n$ facets. By Prop. 6.3 in \cite{every-polytope-is-an-intersection-of-a-simplex-with-affine-space}, any polytope in $\RR^{n-d-1}$ with $n$ facets is combinatorially equivalent  to an intersection of $\Delta_{n-1}$ with an affine space of co-dimension $d$. The same is true for a $(n-d-1)$-dimensional polytope with less than $n$ facets, as we could first intersect $\Delta_{n-1}$ with an affine hyperplane to obtain $\Delta_{n-2}$, and apply induction. That is, we may write our polytope $P$ as
$$P=\left\{x\in\RR^n:Mx=b, \sum_{i=1}^nx_i=1, x\geq 0\right\}$$
where $M=(m_{ij})\in\RR^{d\times n}$ and $b\in\RR^d$. 
Changing coordinates on $\RR^n$, such that $x_n=1$, we may re-write our polytope as
$P=\{x\in\RR^n: Nx=0, x_n=1, x\geq 0\}.$
Here, $N$ is the $d\times n$ matrix obtained from $M$ by subtracting $(m_{in},\ldots, m_{in},b_i)$ from the $i$th row. The cone of $P$ is then $C(P)=\{x\in\RR^n: Nx=0, x\geq 0\}$. Let $x'=(x_1',\ldots,x_n')\in C(P)$. Scaling the $i$th column of $N$ by $x_i'$, we get a new matrix $N'$. Then the cone $C'(P)=\{x\in\RR^n: N'x=0, x\geq 0\}$ contains the all-ones vector. This guarantees that each row of $N'$ sums to $0$, and letting $B:=(N')^T$, we see that the cone $C'(P)$ is equal to the cone of the logarithmic Voronoi polytope at an interior point of a model associated to $B$. As $B$ is an $n\times d$ matrix, this model is $d$-dimensional in $\Delta_{n-1}.$ Thus, $P$ is combinatorially equivalent to a logarithmic Voronoi polytope, as desired. \end{proof}

\section{On the boundary}\label{boundary}
In this section we study logarithmic Voronoi polytopes at the points of a linear model that lie on the boundary of the simplex, where the log-likelihood function is undefined. The next example demonstrates that the combinatorial type of logarithmic Voronoi polytopes at the points on the boundary of $\Delta_{n-1}$ will depend on the positioning of the linear model inside the simplex. Namely, if the intersection of the affine linear space defining the model with $\Delta_{n-1}$ is not general, logarithmic Voronoi polytopes at the boundary points will degenerate.

The definition of the log-likelihood function can be extended to the boundary of the simplex by considering each boundary component of the model as a linear model inside a smaller simplex. Namely, let $\M$ be a $d$-dimensional linear model inside $\Delta_{n-1}$, and let $f$ be a face of $\M$ that lies on the boundary of $\Delta_{n-1}$. Then the relative interior of $f$ lies in the interior of some $\Delta_{k-1}$, which is on the boundary of $\Delta_{n-1}$. We may then treat $f$ as its own linear model inside $\Delta_{k-1}$, and the log-likelihood function is defined for all interior points of $f$.

\begin{example}
Consider a polytope, combinatorially isomorphic to the 3-dimensional cube. According to Proposition \ref{every-polytope-appears}, this polytope appears as a logarithmic Voronoi cell at an interior point on some 2-dimensional linear model in $\Delta_5$. One such model $\M$ is given by
$$B=\begin{bmatrix}
-10 & -2 & -4 & 6 & 6 & 4 \\
3 & 2 & 1 & -1 & -2 & -3
\end{bmatrix}^T\text{ and } c=(1/12,1/3,1/6,1/12,1/6,1/6).$$
It is a triangle, whose vertices are parametrized by $(1/168, -1/21),
(1/24, 1/6),$ and \\$(-1/24, -1/9)$. The logarithmic Voronoi polytopes at the interior points of $\M$ are combinatorially equivalent to the 3-dimensional cube. The vertices are\\

\setlength{\tabcolsep}{1pt}
\begin{adjustbox}{width=1\textwidth}
\begin{tabular}{rcccccccl}
$($&$0,$& $0,$& $16x - 4y + 2/3,$& $-8x + 4/3y + 1/9,$& $-8x + 8/3y + 2/9,$& $0$&$)$\\

$($&$0,$ &$ 0,$& $84/5x - 21/5y + 7/10,$& $-72/5x + 12/5y + 1/5,$ &$ 0,$ &$ -12/5x + 9/5y + 1/10$&$)$\\

$($&$0,$ & ${4}/{3}x - {4}/{3}y + {2}/{9},$ & ${32}/{3}x - {8}/{3}y + {4}/{9},$& $0,$& $-12x + 4y + {1}/{3},$& $0$&$)$\\

$($&$0,$& $4x - 4y + {2}/{3},$&$ 2x - 1/2y + 1/12,$&$ 0,$ &$0,$&$ -6x + 9/2y + 1/4$&$)$\\

$($&$40/17x - 12/17y + 1/51,$ &$ 72/17x - 72/17y + 12/17,$&$ 0,$&$ 0, $& $0,$ &$ -112/17x + 84/17y + 14/51$&$)$\\

$($&$30x - 9y + 1/4,$& $0,$ & $0,$&$ -6x + y + 1/12, $& $-24x + 8y + 2/3,$ &$0 $&$)$\\

$($&$240/11x - 72/11y + 2/11,$& $12/11x - 12/11y + 2/11,$& $0,$& $0,$& $-252/11x + 84/11y + 7/11,$& $0$&$)$\\

$($&$35x - 21/2y + 7/24,$&$ 0,$& $0,$& $-27x + 9/2y + 3/8, $&$0,$&$ -8x + 6y + 1/3$&$)$
\end{tabular}
\end{adjustbox}









for the parameters $(x,y)$. Given a point in $\M$ on the boundary of $\Delta_{5}$, parametrized by $(\hat{x},\hat{y})$, the vertices of its logarithmic Voronoi polytope are obtained by plugging $(\hat{x},\hat{y})$ into the equations above. One checks that at all the points on the boundary of $\M$, the logarithmic Voronoi polytopes are also combinatorially equivalent to the 3-dimensional~cube. 

On the other hand, consider the model given by 
$$B=
\left[\begin{array}{rrrrrr}
-10 & -3 & -20 & 6 & 6 & 21 \\
1 & 3 & 2 & -1 & -3 & -2
\end{array}\right]^T\text{ and }c=(1/6,1/12,1/3,1/6,1/12,1/6).$$ It is a quadrilateral in $\Delta_5$ with the vertices parametrized by $(1/153, -1/68),
(2/171, 3/76),\\
(-5/324, 1/81),$ and 
$(-7/288, -11/144).$ The logarithmic Voronoi polytopes at the interior points are also combinatorially equivalent to the 3-dimensional cube. However, at the vertex parametrized by $(-5/324, 1/81)$, the logarithmic Voronoi polytope is no longer a cube: it degenerates to a 2-dimensional quadrilateral. This is explained by the fact that the vertex lies on a 2-dimensional face of the simplex (as opposed to a 3-dimensional face).
\end{example}
In general, whenever each vertex of a $d$-dimensional linear model lies on a $(n-d-1)$-dimensional face of $\Delta_{n-1}$, the combinatorial type of the logarithmic Voronoi cell at a boundary point is the same as at the interior points. Before proving this result, we first fix some notation.

\textbf{Notation:} Let $\M=\{c-Bx:x\in\Theta\}$ and let $z$ be a cocircuit of $B$ with support  $S$ such that $\sum_{i=1}^nz_if_i(x)=1$, where $f:\Theta\to\RR^n$ is a parametrization of $\M$. Let $V_z(x):\Theta\to \RR^n$ be the vertex of the logarithmic Voronoi polytope determined by $z$, as a function of $x\in\Theta$. That is, $V_z(x)=\big(z_1(c_1-\langle b_1, x\rangle), \ldots, z_n(c_n-\langle b_n, x\rangle)\big)\in\Delta_{n-1}.$ If $w=f(\hat{x})\in\M$ is a point on the boundary of the simplex, then the vertices of the logarithmic Voronoi polytope at $w$ are given as limits of the vertices $V_z(y^{(i)})$ where $\{f(y^{(i)})\}$ is a sequence of interior points converging to $w$. Let $M$ be the $n\times (d+1)$ matrix obtained by concatenating $c_i$ to the $i$th row of $B$, for all $i\in[n]$. If $U$ and $V$ are two sets of the same cardinality in $[n]$ and $[d+1]$, respectively, we denote by $M_{U,V}$ the submatrix of $M$, whose rows are indexed by $U$ and whose columns are indexed by $V$. We define $B_{U,V}$ similarly. Assume, without loss of generality, that the last $k$ columns of $M_{S,[d+1]}$ are linearly independent. We have the following technical~lemma. 

\begin{lemma}\label{less-technical-lemma}
Let $v=f(\hat{x})$ be a vertex of $\M$ with support $I$ and let $z$ be a cocircuit of $B$. The $i$th coordinate of $V_z(\hat{x})$ is zero if and only if $\det M_{\left([n]\setminus I\right)\cup\{i\},[d+1]}=0$.
\end{lemma}
\begin{proof} Since $\M$ is $d$-dimensional, each vertex of the model is determined by the vanishing of precisely $d$ coordinates, i.e.
$c_i=\langle b_i,\hat x\rangle\;\forall i\in [n]\setminus I\text{ and } c_i>\langle b_i, \hat{x}\rangle\;\forall i\in I.$
Without loss of generality, assume $I = \{d+1,\ldots, n\}$, so $v$ is determined by the vanishing of the first $d$ coordinates. Then $\hat{x}=(\hat{x}_1,\ldots,\hat{x}_d)$ is a solution to the linear system $B_{[d],[d]}x=(c_1,\ldots, c_d)$. We may assume $\det B_{[d],[d]}\neq 0$. By Cramer's rule, we may then write $\hat{x}_i=(-1)^{d-i}\frac{\det M_{[d],[d+1]\setminus\{i\}}}{\det B_{[d],[d]}}$ for all $i\in[d]$. Let $S$ denote the support of the cocircuit $z$ and suppose $|S|=k$. Let $z'$ be the projection of $z$ onto its support. Since $z$ is a cocircuit of $B$ such that $\sum_{i=1}^n z_i(c_i-\langle b_i,x\rangle)=1$, it satisfies the equation $c_1z_1+\ldots+c_nz_n=1$. Thus,
$z'$ is a solution to the system $yM_{S, [d+1]}=(0,\ldots,0,1)\in\RR^{d}$. If $d+1\geq k$, then $d+1-k$ equations in this system must be redundant. From our assumption, the first $d+1-k$ equations are redundant, so removing them, we get a $k\times k$ linear system with a unique solution. Using Cramer's rule again, we find that for any $i\in S$, $z_{i}=(-1)^{k+i'}\frac{\det B_{S\setminus\{i\},[d]\setminus[d+1-k]}}{\det M_{S,[d+1]\setminus[d+1-k]}}$, where $i'$ is the index of $z_i$ in $z'$. If $i\notin S$, the $i$th coordinate of $V_z(\hat{x})$ is 0. If $i\in S$, we have the $i$th coordinate of $V_z(\hat{x})$ is given by
$$\msmall{\frac{\det B_{S\setminus\{i\},[d]\setminus[d+1-k]}}{\det M_{S,[d+1]\setminus[d+1-k]}B_{[d],[d]}}\left[c_i\det B_{[d],[d]}-\left((-1)^{d-1}b_{i1}\det M_{[d],[d+1]\setminus\{1\}}+\ldots+b_{id}\det M_{[d],[d+1]\setminus\{d\}}\right)\right]}.$$
Note the expression in square brackets is $(-1)^{k+i'}\det M_{[d]\cup\{i\},[d+1]}$, so the $i$th coordinate of $V_z(\hat{x})$ vanishes if and only if $\det M_{[d]\cup\{i\},[d+1]}=0$.
The case $d+1<k$ is not possible, as it would imply the existence of a cocircuit whose support is strictly contained in $S$. 
\end{proof}

\begin{theorem}\label{linear-boundary}
Let $\M$ be a $d$-dimensional linear model obtained by intersecting the affine linear space $L$ with $\Delta_{n-1}$. Let $w\in \M$ be a point on the boundary of the simplex. If $L$ intersects $\Delta_{n-1}$ transversally, then the logarithmic Voronoi polytope at $w$ has the same combinatorial type as those at the interior points of $\M$.
\end{theorem}
\begin{proof}
It suffices to show that the combinatorial type of the logarithmic Voronoi polytopes at the vertices of the model is the same as at the interior points. Let $v=f(\hat{x})$ be a vertex of $\M$ and without loss of generality assume that it has support $\{d+1,\ldots,n\}$. By Lemma \ref{less-technical-lemma}, if
$i\in S\cap \{d+1,\ldots,n\}$, the logarithmic Voronoi vertex $V_z(\hat{x})$ degenerates to the vertex with 0 in the $i$th coordinate if and only if $\det M_{[d]\cup\{i\},[d+1]}=0$. This condition translates to $v$ lying on a face of $\Delta_{n-1}$ of dimension less than $n-d-1$, namely the one spanning the affine space $\{x\in\RR^n: x_j=0\text{ for all }j\in[d]\cup\{i\}\}$. This means that the affine space $L$ does not intersect $\Delta_{n-1}$ transversally, a contradiction. Thus, the logarithmic Voronoi polytope at any vertex of $\M$ has the same combinatorial type as at the interior points.
\end{proof}

The next example gives a concrete formula for the vertices of logarithmic Voronoi polytopes when the linear model is one-dimensional. The compact description follows from the fact that cocircuits are easy to compute in this case. A one-dimensional model will intersect the simplex transversally if and only if the $1\times n$ matrix $B$ has no repeated entries.

\begin{example}[$d=1$]
Let $\M=\{c-Bx:x\in\Theta\}$ be a 1-dimensional linear model inside the simplex $\Delta_{n-1}$. Let $B=[b_1,\ldots, b_m,b_{m+1},\ldots b_n]^T$, and without loss of generality assume $b_i>0$ for $i=1,\ldots, m$ and $b_i<0$ for $i=m+1,\ldots,n$. Then $\Theta\subseteq \RR$ is a closed interval $[x_\ell,x_r]$, where $x_\ell=c_\ell/b_\ell$ for some $\ell>m$ and $x_r=c_r/b_r$ for some $r\leq m$. Rotating the simplex, if necessary, we may ensure that $r=1$. Note that any positive cocircuit $z$ of $B$ has support $\{i,j\}$ of size two, where $b_{i}>0$ and $b_{j}<0$. So, we find the logarithmic Voronoi polytope at $x_r$ is the polytope at the boundary of $\Delta_{n-1}$ with the vertices
$$\{e_j: b_j<0\}\cup \Bigg\{\frac{(c_i-b_i(c_1/b_1))b_j}{b_jc_i-b_ic_j}e_i-\frac{(c_j-b_j(c_1/b_1))b_i}{b_jc_i-b_ic_j}e_j:{\substack{i\neq 1, \\b_i>0,\\ b_j<0}}\Bigg\}.$$
The logarithmic Voronoi polytope at $x_\ell$ is described similarly. Figure \ref{figure:tetrahedra-examples-d=1} plots logarithmic Voronoi polytopes at sampled points on 1-dimensional linear models in general position given by $c=(1/4,1/4,1/4,1/4)$,  $B=[1, -5, 3, 1]^T$ and $B=[-2, -1, 1, 2]^T$, respectively.

\comm{
\begin{figure}[H]
    \centering
    \includegraphics[width=0.45\textwidth]{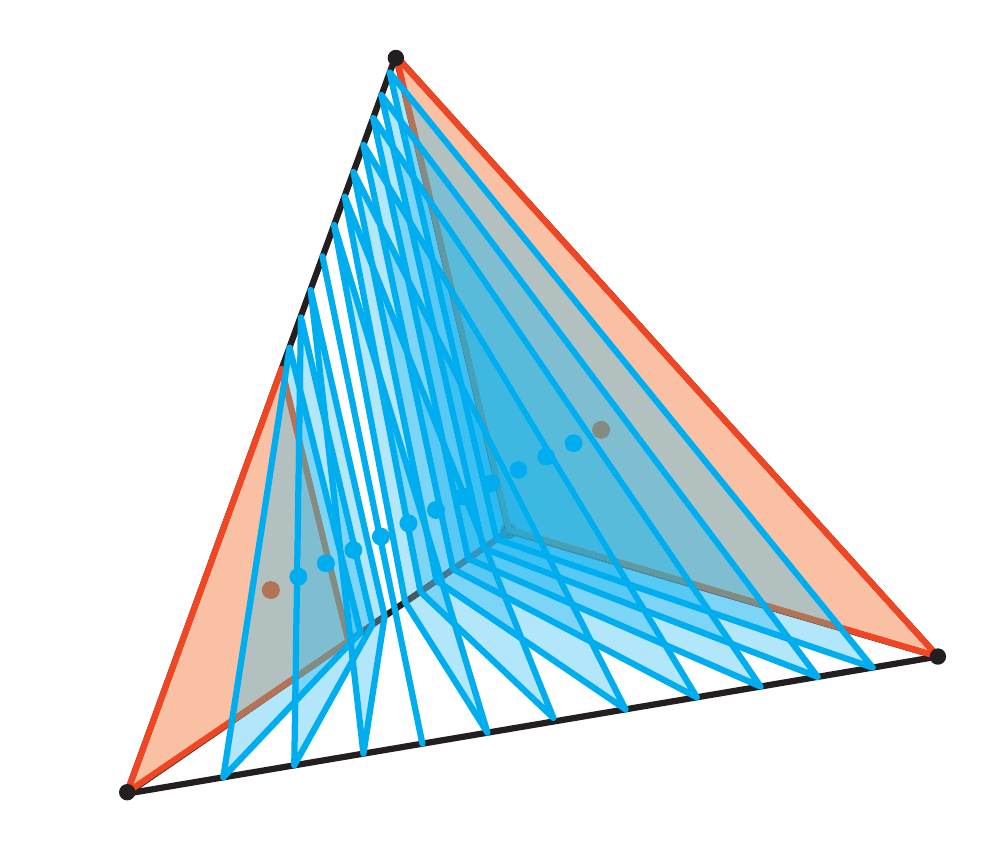} \includegraphics[width=0.45\textwidth]{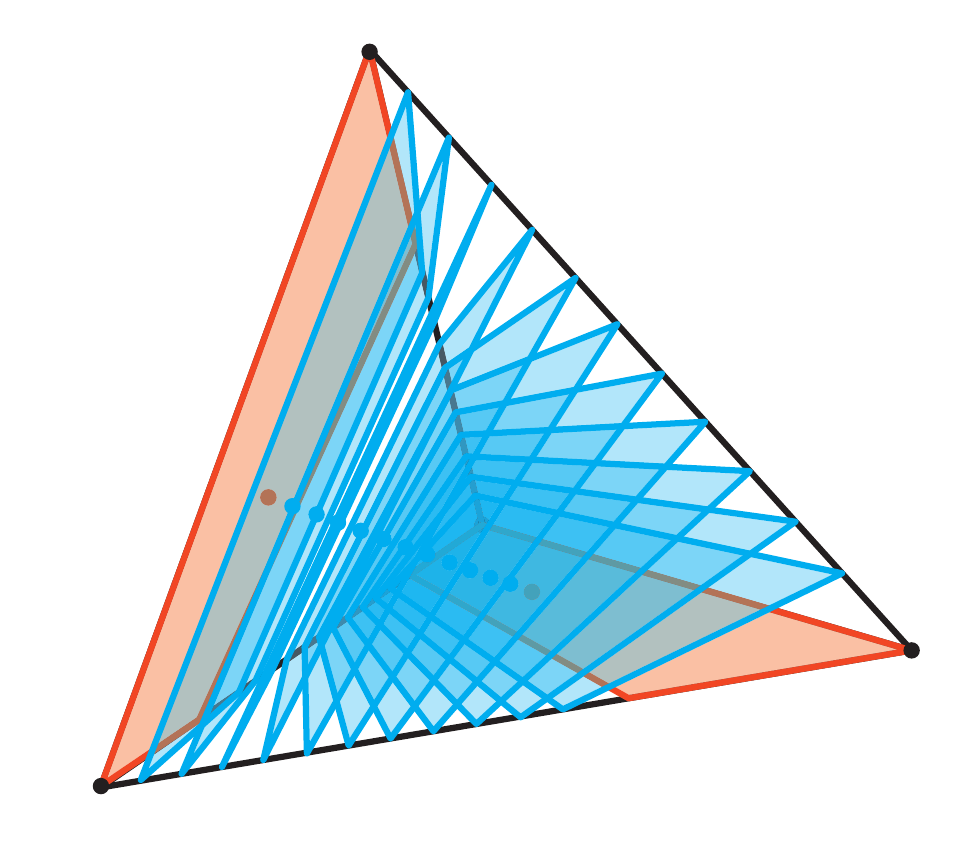} 
    \caption{Logarithmic Voronoi polytopes at points on linear models given by $c=(1/4,1/4,1/4,1/4)$ and $B=[1, -5, 3, 1]^T$ (left) and $B=[-2, -1, 1, 2]^T$ (right).}
    \label{figure:tetrahedra-examples-d=1}
\end{figure}
}

\end{example}

\begin{example}[Moduli spaces]\label{moduli}
The moduli space $\M_{g,m}$ is the space of genus $g$ curves with $m$ marked points. The moduli space $\M_{0,m}$ is the space of $m$ marked points in $\PP^1$ and can be viewed as a linear statistical model of dimension $m-3$ inside the simplex $\Delta_{n-1}$, where $n=m(m-3)/2$. The connection between particle physics and algebraic statistics via moduli spaces has been studied in \cite{modulispaces}. The model $\M_{0,6}$ is a 3-dimensional linear model (a tetrahedron) inside the 8-dimensional simplex. It is parametrized by
$$(x,y,z)\mapsto \left(\frac{5x}{9} ,\,\frac{y}{3} ,\,\frac{z}{9} ,\,\frac{y-x}{9} ,\,\frac{z-x}{9},\,\frac{y-z}{9} ,\,\frac{1-x}{3} ,\,\frac{1-y}{3} ,\,\frac{1-z}{3},\right).$$
Logarithmic Voronoi polytopes at the interior points on this model are 5-dimensional with the $f$-vectors $(7,19,26,19,7)$.

The affine space defining this model does not intersect the simplex transversally; furthermore, none of the four vertices lie on the interior of a $5$-dimensional face of $\Delta_{8}$. Two of the vertices lie on 4-dimensional faces of $\Delta_8$ and the other two vertices lie on 2-dimensional faces of $\Delta_8$. The logarithmic Voronoi polytopes at these vertices degenerate into 4-dimensional and 2-dimensional polytopes, respectively. These polytopes are the entire faces of $\Delta_8$ that contain the corresponding vertices in their relative interior.
\end{example}

\section{Partial linear models}\label{partial}
A \textit{partial} linear model of dimension $d$ is a statistical model given by a $d$-dimensional polytope inside the probability simplex $\Delta_{n-1}$, such that not all facets of the polytope lie on the boundary of the simplex.

Let $\M$ be a partial linear model of dimension $d$ inside $\Delta_{n-1}$. The intersection of the affine span of the polytope $\M$ with the simplex $\Delta_{n-1}$ is a $d$-dimensional linear model $\M'$. We say $\M'$ \textit{extends} $\M$. As in Section \ref{prelim}, $\M'=\{c-Bx:x\in\Theta'\}$
for some appropriate $c,B$, and parameter space $\Theta'\subseteq \RR^d$. Since $\M'$ extends $\M$, it follows that we may also write 
$$\M=\{c-Bx:x\in\Theta\}$$
for some $\Theta\subseteq\Theta'$. Note that both $\Theta$ and $\Theta'$ are polytopes.

\begin{theorem}\label{partial-interior}
Let $\M$ be a partial linear model of dimension $d$ with extension $\M'$. If $p$ is a point in the relative interior of $\M$, then $\log\Vor_{\M}(p)=\log\Vor_{\M'}(p)$.
\end{theorem}
\begin{proof}
We show these sets are contained in each other. First, let $u\in\log\Vor_{\M'}(p)$. Then $\ell_u(x)$ is maximized at $p$ in $\M'$. Since $\M\subseteq\M'$, and $p\in\M$ as well, it follows that $\ell_u(x)$ will also be maximized at $p$ in $\M$. Thus, $u\in\log\Vor_{\M}(p)$. 

Now, let $u\in\log\Vor_{\M}(p)$. If $u\notin\log\Vor_{\M'}(p)$, then over $\M'$, $\ell_u(x)$  is maximized at some other point $q\in\M'\setminus\M$. Then the line segment $[p,q]$ must intersect the boundary of the model $\M$. Note that any point on $[p,q]$ can be written as $a_x=(1-x)p+xq$ for some $x\in[0,1]$. Recall that the
log-likelihood function $\ell_u$ is strictly concave on the simplex and hence on any convex subset of the simplex, such as our model $\M'$. So, for any $x\in(0,1)$, we have
$$
\ell_u(a_x)=\ell_u((1-x)p+xq)>(1-x)\ell_u(p)+x\ell_u(q)>\ell_u(p),$$
where the last inequality follows from the assumptions $\ell_u(q)>\ell_u(p)$ and $x>0$. But since $p$ is in the relative interior of the polytope $\M$, this implies that there is  another interior point $r$ on the line segment $[p,q]$ such that $\ell_u(r)>\ell_u(p)$. This is a contradiction to $u$'s inclusion in $\log\Vor_{\M}(p)$. Therefore, $u\in\log\Vor_{\M'}(p)$, as~desired.
\end{proof}

The theorem above tells us that the logarithmic Voronoi polytopes at points in the interior of the polytope $\M$ are the same as those in the full linear extension $\M'$. The points $u\in\Delta_{n-1}$ that are not in $\log\Vor_\M(p)$ for any $p$ in the interior of $\M$ will be mapped to the points on the boundary of $\M$ via the MLE map. Note that for each point $q$ on the boundary of $\M$, we still have $\log\Vor_{\M'}(q)\subseteq\log\Vor_{\M}(q)$. However, in general, this containment will be strict.

Given a facet $F$ of $\M$, let $p$ be a point in the relative interior of $F$ (i.e. $p$ does not lie on any lower-dimensional face). Treating $F$ as its own partial linear model with extension $F'$ inside $\Delta_{n-1}$, we know that $\log\Vor_{F}(p)=\log\Vor_{F'}(p)$ is an $(n-d)$-dimensional polytope. Moreover, it is clear that $\log\Vor_{\M'}(p)\subseteq \log\Vor_{F}(p)$. Observe that $\log\Vor_{\M'}(p)$ has dimension $n-d-1$ and the boundary of this polytope is included in the boundary of $\log\Vor_{F}(p)$, since these logarithmic Voronoi polytopes are the intersections of affine linear spaces with the simplex. Hence, $\log\Vor_{\M'}(p)$ divides the polytope $\log\Vor_{F}(p)$ into two $(n-d)$-dimensional polytopes. Since $p$ is on the boundary of the polytope $\M$, one of those polytopes will intersect the relative interior of $\M$.

\textbf{Notation:} Denote the two polytopes defined above by $Q_p$ and $\overline{Q}_p$. Assume $\overline{Q}_p$ is the polytope that intersects the relative interior of $\M$.

\begin{figure}[H]
    \centering
    \includegraphics[width=0.35\textwidth]{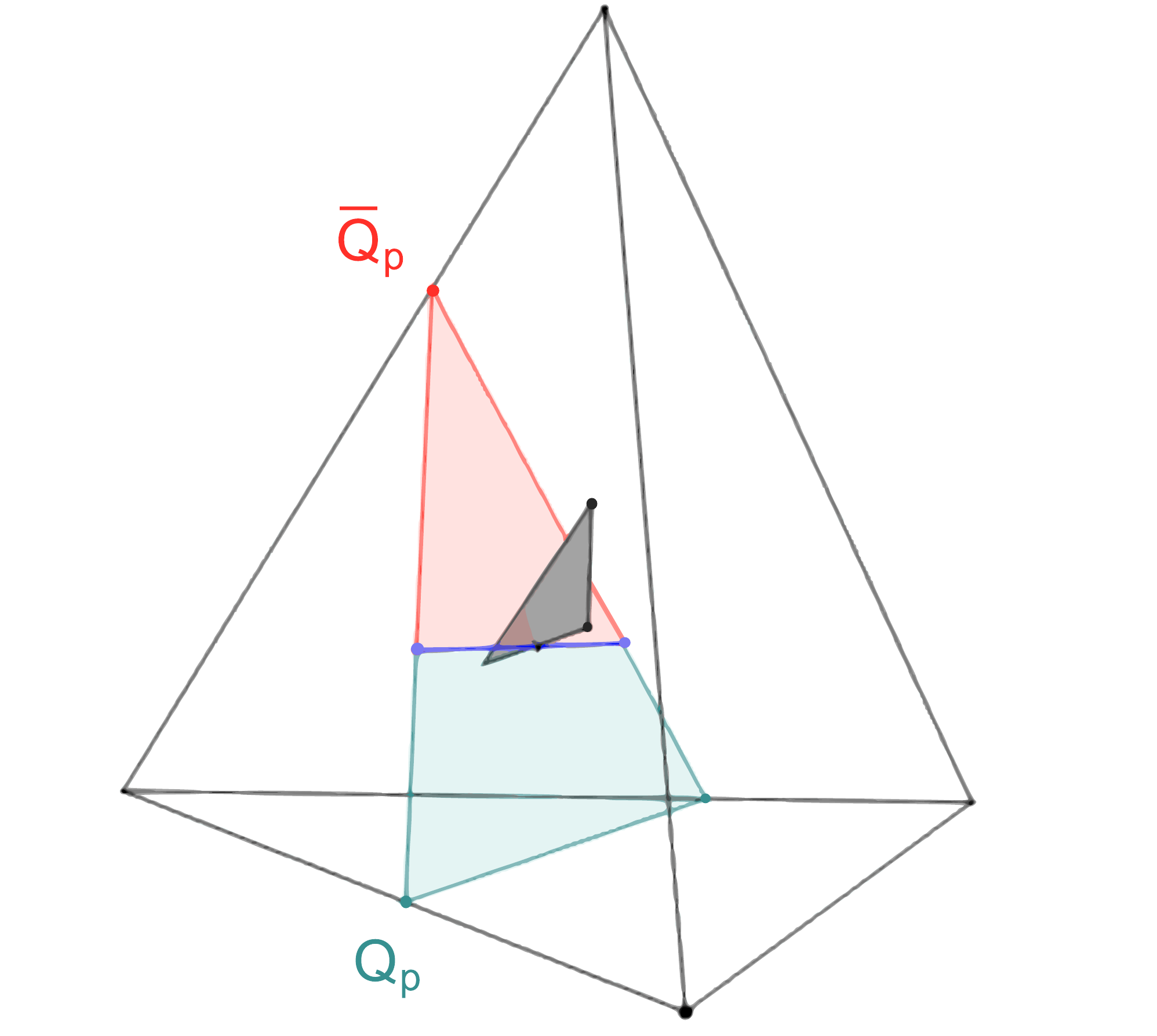}
    \includegraphics[width=0.45\textwidth]{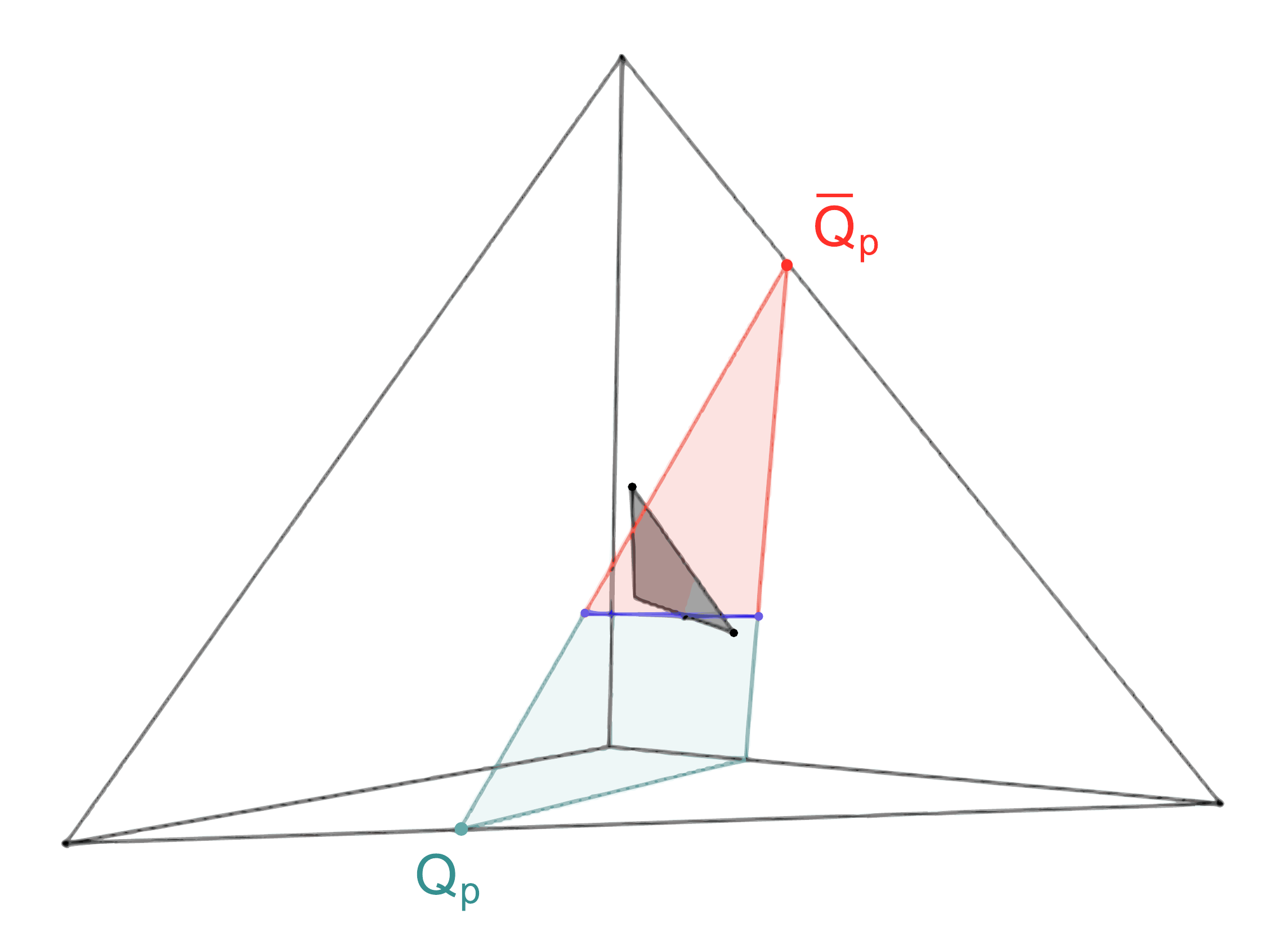}

    \caption{Polytopes $Q_p$ and $\overline{Q}_p$  for a point on a facet of a 2-dimensional model in $\Delta_3$. }
    \label{figure:Q-polytopes}
\end{figure}

\begin{theorem}\label{partial-facets}
Let $p$ be a point in the relative interior of some facet $F$ of $\M$. Let $Q_p$ be as above. Then $Q_p=\log\Vor_\M(p)$.
\end{theorem}
\begin{proof}
Let $u\in \log\Vor_\M(p)$. Since $F\subseteq \M$ and $p$ is in the interior of $F$, we have $$u\in \log\Vor_\M(p)\subseteq\log\Vor_{F'}(p)=\log\Vor_{F}(p)=Q_p\cup \overline{Q}_p.$$
If $u\notin Q_p$, then $u\in \overline{Q}_p\setminus Q_p$. In particular, $u\notin\log\Vor_{\M'}(p)$. But since $\log\Vor_{\M'}(p)\subseteq \log\Vor_\M(p)$ and $\log\Vor_\M(p)$ is convex, we also have $\conv\{u,\log\Vor_{\M'}(p)\}\subseteq\log\Vor_\M(p)$. But then by construction of $\overline{Q}_p$, the convex hull above will contain an interior point in $\M$. This is a contradiction, as logarithmic Voronoi polytopes at distinct points on the model cannot intersect; thus $u\in Q_p$.

For the other direction, note that the polytope $Q_p$ has two types of points: the points in the polytope $\log\Vor_{\M'}(p)$ and the points not in $\log\Vor_{\M'}(p)$. Since $\log\Vor_{\M'}(p)\subseteq \log\Vor_{\M}(p)$, it suffices to show that $w\in \log\Vor_{\M}(p)$ for each point $w\in Q_p$ of the second type. We show that $w\in \log\Vor_{\M}(p)$. Note that we may assume that $w$ in in the interior of $\Delta_{n-1}$, since taking the closure would preserve the containment.  Note that $\ell_w(x)$ is a strictly concave function on the simplex, so its super-level sets
$$C_{\alpha}=\{x\in\Delta_{n-1}:\ell_w(x)\geq \alpha\}$$
are convex $(n-1)$-dimensional sets. Since the maximum of $\ell_w(x)$ on $F'$ is achieved at $p$, we know that it is given by $\ell_w(p)=\max\{\alpha: C_\alpha\cap F\neq \varnothing\}$. Note that $F'$ divides the linear extension $\M'$ into two polytopes, $S_1$ and $S_2$, where $S_1$ is the polytope containing the model $\M$. If $w\notin \log\Vor_\M(p)$, then $\Phi_\M(w)=q\neq p$, where $q\notin F$. So, $q\in S_1$ lies on some other facet of $\M$. Moreover, $m\notin F'$, since $p$ is the maximizer over $F'$ and $\ell_w(m)>\ell_w(p)$.\\
\textbf{Case 1:} Suppose $m\in S_1$. Note that $R=\bigcup_{r\in F'}\log\Vor_{\M'}(r)$ is an $(n-2)$-dimensional hypersurface inside $\Delta_{n-1}$, obtained by intersecting a ruled hypersurface in $\RR^n$ with the simplex. Thus, $R$ subdivides the simplex into two full-dimensional parts. By construction, $w$ and $m$ are on different sides of $R$. Since logarithmic Voronoi cells are convex sets, the line $[w,m]\subseteq \log\Vor_{\M'}(m)$, and this line intersects $R$. This is a contradiction, since logarithmic Voronoi cells at two distinct points on the same model cannot intersect. \\
\textbf{Case 2:} If $m\in S_2$, then since $\ell_w(m)>\ell_w(p)$ and $\ell_w(q)>\ell_w(p)$, there exists some $\alpha$ such that $C_\alpha\subsetneq C_{\ell_w(p)}$ and such that $C_\alpha$ contains $q$ and $m$, but does not contain $p$. Since super-level sets are convex, the line segment $[q,m]$ between $q$ and $m$ is contained in $C_\alpha$. But since $q\in S_1$ and $m\in S_2$, the line $[q,m]$ intersects $F'$ in some point $s\neq p$. But then $\ell_w(s)>\ell_w(p)$, a contradiction.\\
We conclude that $w\in\log\Vor_{\M}(p)$. Since logarithmic Voronoi cells are closed sets, the closure of all such points $w$ is also contained in $\log\Vor_{\M}(p)$, and the conclusion follows.
\end{proof}

Now suppose $F$ is a face of $\M$ of dimension $d-k$ for some $k\geq 2$. Then $F$ is the intersection of at least $k$ faces of dimension $d-k+1$. Denote those faces by $\{G_1,\ldots, G_m\}$, where $m\geq k$. For each $i\in [m]$, $\log\Vor_{G_i'}(p)$ subdivides $\log\Vor_{F'}(p)$ into two polytopes. Exactly one of these polytopes will intersect the face $G_i$ at an interior point; call such polytope $\overline{Q}_i$. Call the other polytope $Q_i$. We present the following conjecture.

\begin{conj}\label{partial-boundary}
Let $p$ be a point in the relative interior of the face $F$ of $\M$. Then $\bigcap_{i\in [m]}Q_i=\log\Vor_{\M}(p)$. In particular, if $\M$ is in general position, $\dim \log\Vor_{\M}(p)=(n-1)-\dim F$.
\end{conj}

\begin{example}
Let $d=2$, $n=4$, and consider the model $\M$ defined as the convex hull of $\left(\frac{1}{5}, \frac{1}{5}, \frac{1}{5}, \frac{2}{5}\right), \left(\frac{1}{5}, \frac{1}{5}, \frac{2}{5}, \frac{1}{5}\right)$, and $\left(\frac{1}{4}, \frac{1}{4}, \frac{1}{4}, \frac{1}{4}\right)$. Below we plot the logarithmic Voronoi cells at interior points, edges, and vertices consecutively.
\end{example}

\begin{figure}[H]
    \centering
    \includegraphics[width=0.31\textwidth]{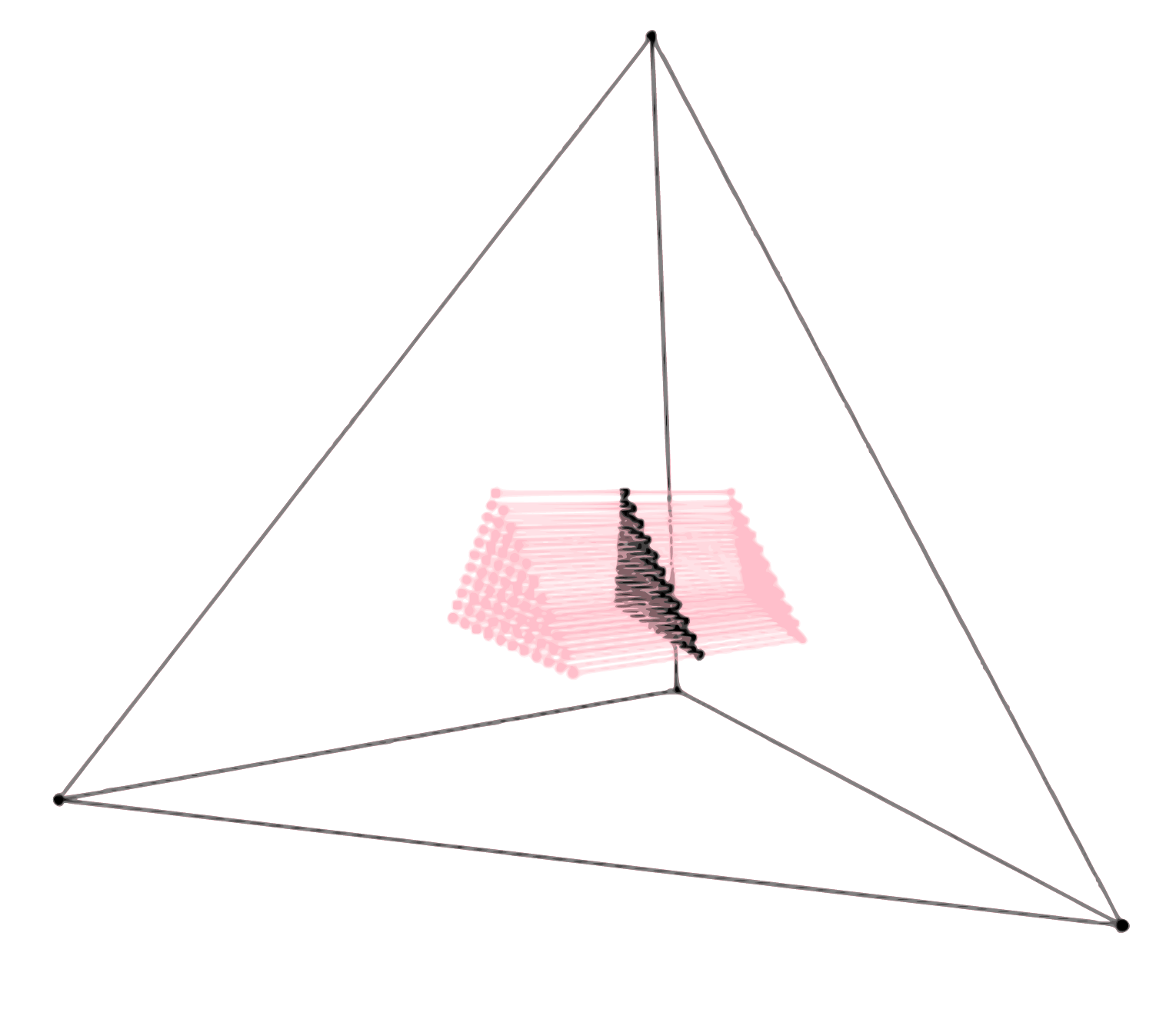}
    \includegraphics[width=0.33\textwidth]{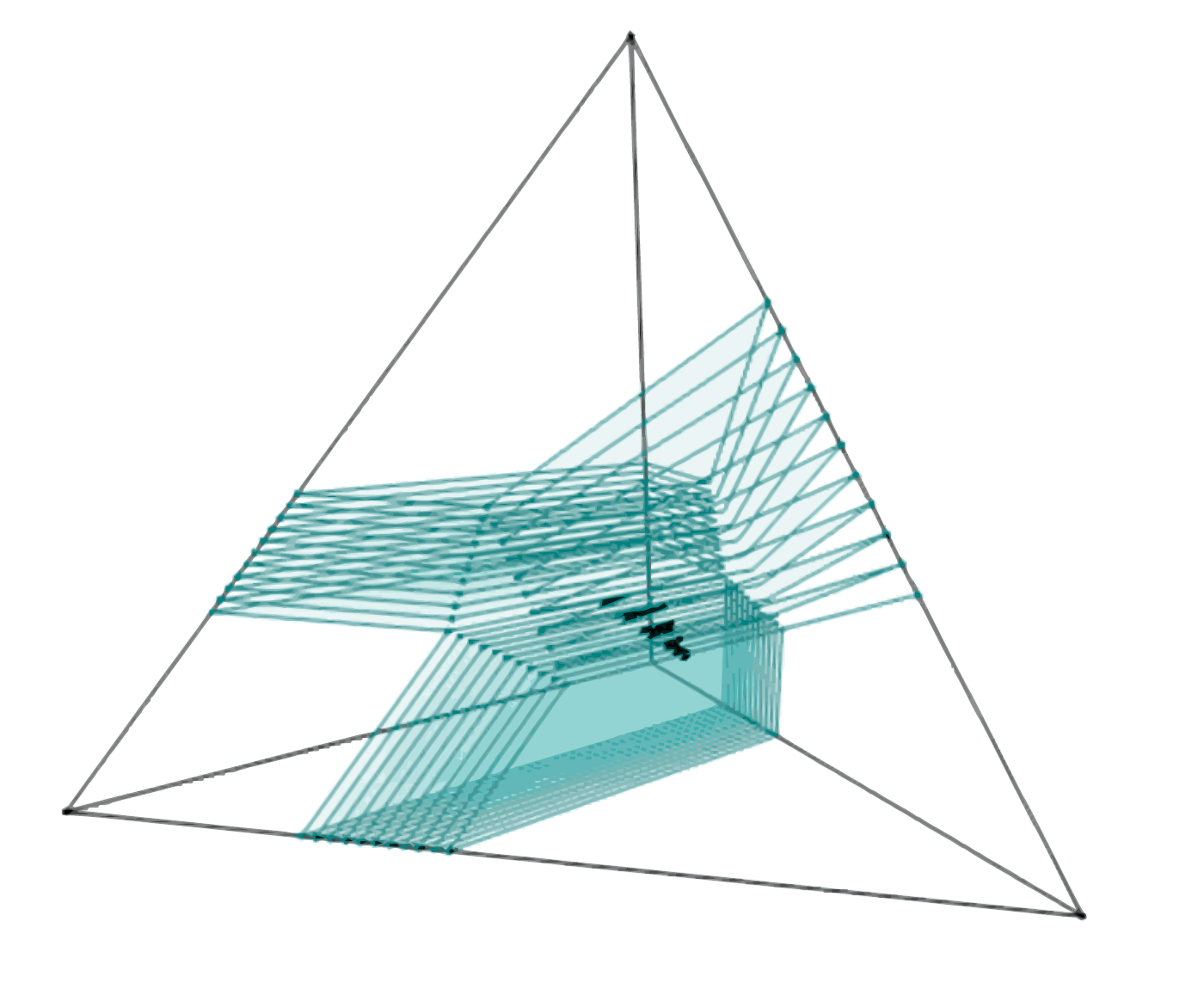}
    \includegraphics[width=0.32\textwidth]{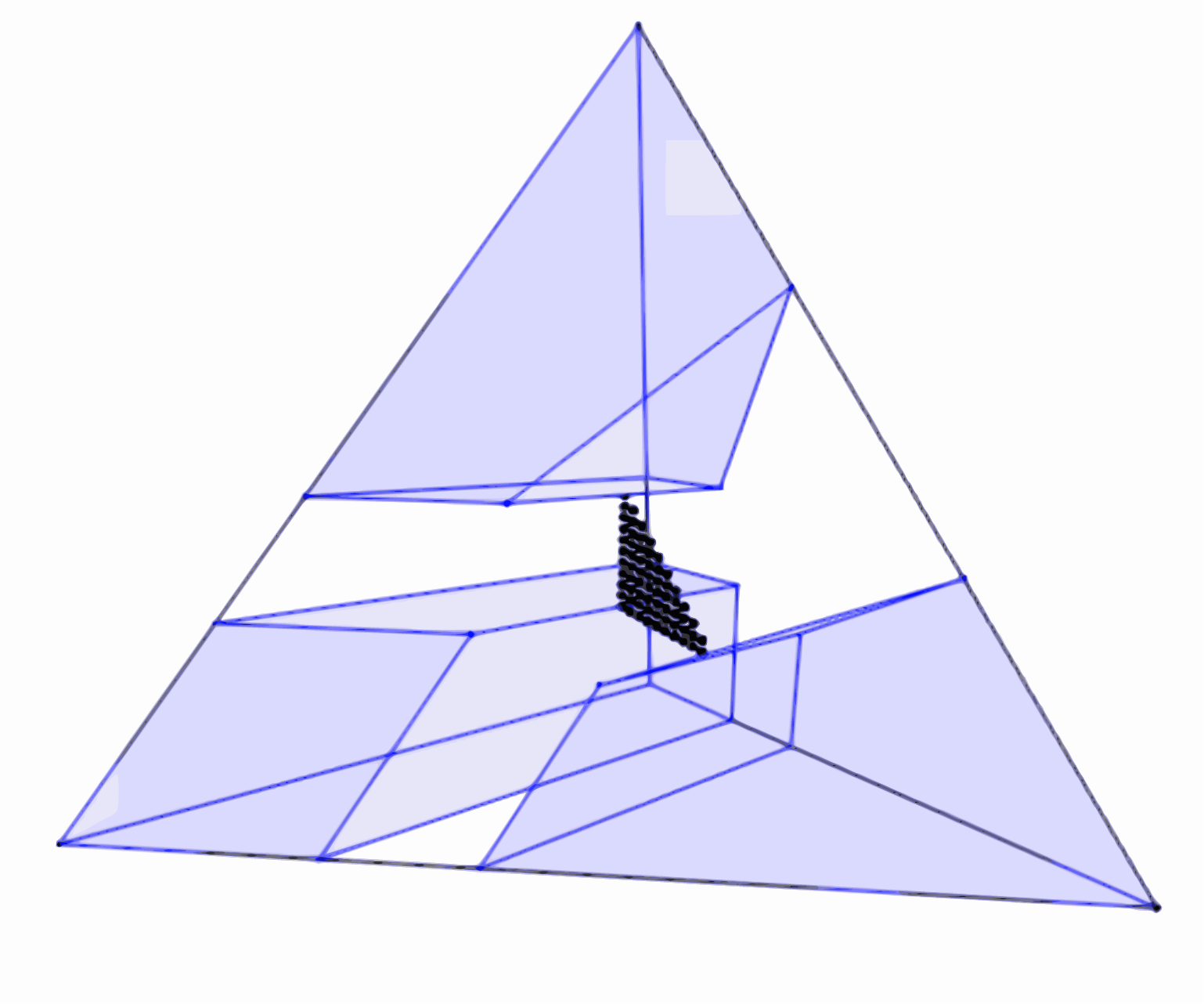} 
    \caption{Logarithmic Voronoi cells at sampled points on the model defined as the convex hull of three points.}
    \label{figure:partial-d2n4}
\end{figure}
\textbf{Acknowledgements}: We thank Serkan Ho\c{s}ten and Bernd Sturmfels for many helpful discussions, suggestions, and comments on the manuscript. We also thank Marie-Charlotte Brandenburg for discussions about polytopes and Thomas Endler for producing Figure~\ref{figure:tetrahedra-examples-d=1}. Finally, we thank the anonymous reviewers for their careful reading of the manuscript and their many insightful comments and suggestions. This material is based upon work supported by the National Science Foundation Graduate Research Fellowship under Grants No. DGE 1752814 and DGE~2146752. 
\bibliographystyle{plain}
\bibliography{references}
\end{document}